\documentclass[12pt,a4paper,oneside]{amsart}
\usepackage{amsmath,amstext,amssymb,amsopn,amsthm,color}

\usepackage[T1]{fontenc}

\theoremstyle{plain}
\newtheorem{thm}{Theorem}[section]
\newtheorem{lem}[thm]{Lemma}

\newtheorem{cor}[thm]{Corollary}
\theoremstyle{definition}

\newtheorem*{rem*}{Remark}

\newcommand{\R}{\mathbb{R}}

\newcommand{\C}{\mathbb{C}}
\newcommand{\E}{\mathbb{E}}

\newcommand{\p}{\mathbb{P}}
\newcommand{\pr}{\mathbb{P}}

\renewcommand{\leq}{\leqslant}
\renewcommand{\geq}{\geqslant}

\newcommand{\lmu}{\stackrel{\mu}{\lesssim}}

\newcommand{\appmu}{\stackrel{\mu}{\approx}}

\def\({\left(}
\def\){\right)}
\def\[{\left[}
\def\]{\right]}
\def\<{\langle}
\def\>{\rangle}

\begin{document}
\title[Exit times densities  of Bessel process]{Exit times densities  of  Bessel process}
\author{Grzegorz Serafin}

\address{ Grzegorz Serafin \\
Department of Mathematics\\ Faculty of Fundamental Problems of Technology\\  Wroc{\l}aw University of Technology\\ ul.
Wybrze{\.z}e Wys\-pia{\'n}skiego 27, 50-370 Wroc{\l}aw\\
Poland
}
\email{ grzegorz.serafin@pwr.edu.pl}


\thanks{The project was funded by the MNiSW grant no. IP2012018472.}

\maketitle

\begin{abstract}
We examine  the  density functions  of the first exit times of the Bessel process  from  the intervals $[0,1)$ and $(0,1)$. First, we express them  by means of the transition density function of the killed process. Using that relationship  we provide precise estimates and asymptotics of the exit time densities. In particular, the results hold for the first exit time of the $n$-dimensional Brownian motion from a ball.
\end{abstract}

\section{Introduction}

The Bessel process  $R^{(\mu)}(t)$  with index $\mu\in\R$ is a well known diffusion on $[0,\infty)$.   For indices of the form $\mu=n/2-1$, $n=1, 2, 3,...$, it  may be described as a norm of $n$-dimensional Euclidean Brownian motion. The process  is also closely related to the  geometric Brownian motion by the Lamperti relation.  Furthermore,  in form of the Bessel-Brownian diffusion, it plays a crucial role in approach to hyperbolic Brownian motion \cite{BTFY, BR} and is used to represent  symmetric stable processes as well \cite{MO}. Finally,  it is  exploit to model a stock prices and has further applications to mathematical finance \cite{D, GY, Y}.

In the context of Bessel process, the killed process was recently of special interest \cite{BM2, BM1, BMR, BR, HM1, HM2, U}.   In the paper, we examine the exit time density of the Bessel process starting from $x>0$ and killed when it reaches a fixed  \mbox{level $a>x$}.   Due to a scaling property it is enough to consider $a=1$. Apparently, there are in the literature noticeably more results concerning starting point greater than the  killing level. Reasons of this disproportion are not clear.  Bessel process killed when exiting $[0,1)$ is interesting not only from probabilistic but also from analytical point of view, since its transition density is well known as a Fourier-Bessel heat kernel. One of obstacles when dealing with this case may be its complex nature. Technical complications appear even for $\mu=-1/2$ (see \cite{PSZ}), it is when the Bessel process corresponds to the one-dimensional Brownian motion. Our aim is to provide asymptotics and uniform estimates of the exit time density in whole range of index $\mu\in\R$. Analogous results are already known in case $x>a$ \cite{BMR, HM1, U}, but are obtained by usage of  completely different methods. Note that presented in the paper result hold, as a special case, for the first exit time of the Brownian motion from a ball and, up to the knowledge of the author, they were not known before.

One of the difficulties to deal with is the  behaviour of the Bessel process at zero, which depend on the value of the index $\mu$. For $\mu\geq0$ the process never riches zero, and for $\mu\leq-1$ we impose that the point zero is killing. In case $\mu\in(-1,0)$ we  may consider one of both: killing or reflecting condition at zero.  Let $T_{a}$ be the first hitting time of the Bessel process at the point $a\geq0$. Additionally, we introduce the first exit time $T_{a,b}=T_a\wedge T_b$ from the interval $(a,b)$, where $0\leq a<x<b$. According to the dual nature of the point zero we define for $\mu>-1$ and reflecting condition at zero
$$q_b^{(\mu)}(t,x)=\p^x\(T_b\in dt\),\ \ \ x\in(0,1), t>0.$$
Furthermore, when $\mu<0$ and zero is killing we denote
$$q_{a,b}^{(\mu)}(t,x,y)=\p^x\(T_{a,b}\in dt;R^{(\mu)}\(T_{a,b}\)=y\),$$
where $x\in(a,b)$, $y\in\{a,b\}$ and $t>0$. In 1980, J.T. Kent yielded in his paper  \cite{K} the following series formula
\begin{equation}\label{eq:q01series}
q_1^{(\mu)}(x,t)=x^{-\mu}\sum_{n=1}^\infty j_{\mu,n}\frac{J_\mu(j_{\mu,n}x)}{J_{\mu+1}(j_{\mu,n})}  e^{-j_{\mu,n}^2t/2}.
\end{equation}
Similar results may be obtained also for the function $q_{0,1}^{(\mu)}(t,x,y)$, which is presented in Section 3. Those formulae are convenient when dealing with large times. We focus on the other case, i.e. when $t$ is small. The above-given sum is oscillating  and, in that case,   the cancellations between the terms really matter in the context of asymptotic behaviour. 

The main idea  of our approach to the problem  is to express the exit times densities in terms of the density of the killed Bessel process. This kind of formulae  seems to be quite general, but the mentioned singularity of the process at zero requires a separate consideration. The obtained representations enable us to use some results and methods provided in \cite{MSZ}, where the density of the Bessel process killed when hitting the level $1$ was estimated. One of the our key tools is an imitation of a reflection principle, which lets us efficiently approximate the  density of the killed process  for some range of parameters. Properties of Bessel  functions are also intensively exploit. 

The paper is organized as follows. In Section $2$ we introduce some notation and gather relevant facts concerning Bessel functions and the Bessel process. Section $3$ is devoted to representations of the densities of the first exit times and theirs estimates. The precise asymptotic behaviour is provided in Section $4$.

\section{Preliminaries}
\subsection{Notation}
For two positive functions $f$ and $g$ we denote  $f\lesssim g$ whenever there exists a constant $c>0$ such that $f<cg$ for all arguments. If $f\lesssim g$ and $g\lesssim f$ we write $f\approx g$. Additionally,  if the constant $c$ depends on an a parameter, we put this parameter over the sign $\lesssim$ or $\approx$. 
 Furthermore, for $x>0$ we use notation $f(x)=O_\mu(x)$, which throughout the paper means that there exist  positive constants $c_1=c_1(\mu)$ and $c_2=c_2(\mu)$, such that $|f(x)|<c_1x$ for $|x|<c_2$.

\subsection{Bessel functions }

The Bessel function of the first kind $J_\mu(z)$ is given by the series
$$J_\mu(z) = \sum_{k=0}^\infty \left(\frac{z}{2}\right)^{\mu+2k}\frac{1}{k!\Gamma(k+\mu+1)}\/,\quad \arg(z)<\pi\/,\quad \mu\in\R.$$
 One can deduce from the definition the asymptotic behaviour of the function  at zero 
\begin{align}\label{eq:Jat0}
\lim_{z\rightarrow 0}J_\mu(z)z^{-\mu}=1/\Gamma\(\mu+1\).
\end{align}
Combining this with the formula  8.451(1) in \cite {GR} we obtain 
\begin{align}\label{eq:estJ}
\left|J_\mu(z)\right|\lmu\frac{z^{\mu}}{(1+z)^{\mu+1/2}}.
\end{align}
 For $\mu>-1$ the  function $J_\mu(z) $ has infinitely many real zeros $(j_{\mu,k})_{k=1}^\infty$, which have the following properties (see \cite{W}, p. 199, 506)
\begin{align}\label{eq:j}
\lim_{k\rightarrow\infty}\frac{j_{\mu,k}}{k}&=\pi,\\\label{eq:Jj}
J_{\mu+ 1}(j_{\mu, k})&=\frac{(-1)^{k+ 1}}\pi\sqrt{\frac{2}{k}}  \(1 + O_\mu\(\frac1k\)\).
\end{align}
The modified Bessel function of the first kind is given by $I_\mu(z)=J_{\mu}(iz)$. The function is positive for $z>0$ and its asymptotic behaviour is given by  (cf. \cite{GR}, formulae  8.445 and 8.451(5))
\begin{align}\label{eq:asympIinfty}
I_{\mu}(z)&=\frac{e^{z}}{\sqrt{2\pi z}}\(1+O_\mu\(\frac1z\)\),\ \ \ \ \ z\rightarrow\infty,\\
\label{eq:asympI0}
I_{\mu}(z)&=\frac{z^\mu}{2^\mu\Gamma\(\mu+1\)}\(1+O_\mu\(z^2\)\),\ \ \ \ \ z\rightarrow0.
\end{align}
This implies
\begin{equation}\label{est:I}
I_{\mu}(z)\appmu \frac{z^\mu}{(1+z)^{\mu+1/2}}e^{z},\ \ \ z>0.
\end{equation}
Furthermore, there are provided in  \cite{L} the following bounds for ratio $I_\mu(y)/I_{\mu}(x)$  of modified Bessel functions of the first kind
\begin{align} \label{ineq:L<}
\frac{I_{\mu}(y)}{I_{\mu}(x)}<\(\frac{y}{x}\)^\mu e^{y-x},\ \ \ \mu>-1/2,\\\label{ineq:L>}
\frac{I_{\mu}(y)}{I_{\mu}(x)}>\(\frac{x}{y}\)^\mu e^{y-x},\ \ \ \mu>1/2.
\end{align}
where $0<x<y$. However, we will need in the sequel results of that kind valid for all  $\mu>-1$. We  prove therefore  inequalities which are slightly weaker but hold in the whole required range of parameters.

\begin{lem}\label{lem:I/I}
For $\mu>-1$ and $0<x<y$ we have
$$
\(\frac{y}{x}\)^{\mu+2} e^{y-x}>\frac{I_{\mu}(y)}{I_{\mu}(x)}>\(\frac{x}{y}\)^{\mu+4} e^{y-x}.
$$
\end{lem}
\begin{proof}
The proof is based on the equality (see \cite{GR} 8.486 1.)
\begin{equation}\label{eq:recI}
I_{\mu}(z)=I_{\mu+2}(z)+\frac{2(\mu+1)}{z}I_{\mu+1}(z),\ \ \ z\in\C-\{0\}, \ \mu\in\R.
\end{equation}
Using it and the formula  (\ref{ineq:L<}) we get  for $\mu>-1$ and  $0<x<y$ 
\begin{align*}
I_{\mu}(y)&=I_{\mu+2}(y)+\frac{2(\mu+1)}{y}I_{\mu+1}(y)\\
&<\(\frac{y}{x}\)^{\mu+2}e^{y-x}I_{\mu+2}(x)+\frac{2(\mu+1)}{y}\(\frac{y}{x}\)^{\mu+1}e^{y-x}I_{\mu+1}(x)\\
&<\(\frac{y}{x}\)^{\mu+2}e^{y-x}\[I_{\mu+2}(x)+\frac{2(\mu+1)}{x}I_{\mu+1}(x)\]\\
&=\(\frac{y}{x}\)^{\mu+2}e^{y-x} I_{\mu}(x),
\end{align*}
 which is equivalent to the first inequality in the thesis. The lower bound will be proved  in two steps, since (\ref{ineq:L>}) works only for $\mu>1/2$ and the formula (\ref{eq:recI}) let us enlarge the range of $\mu$ only by $1$.  First, assume that $\mu>-1/2$. Then, by (\ref{ineq:L>}) and (\ref{eq:recI}), we obtain
\begin{align*}
I_{\mu}(y)&=I_{\mu+2}(y)+\frac{2(\mu+1)}{y}I_{\mu+1}(y)\\
&>\(\frac{x}{y}\)^{\mu+2}e^{y-x}\[I_{\mu+2}(x)+\frac{2(\mu+1)}{y}I_{\mu+1}(x)\]\\
&=\(\frac{x}{y}\)^{\mu+2}e^{y-x}\[I_{\mu}(x)-\frac{y-x}{y}\frac{2(\mu+1)}{x}I_{\mu+1}(x)\]\\
&>\(\frac{x}{y}\)^{\mu+2}e^{y-x}\(1-\frac{y-x}{y}\) I_{\mu}(x)\\
&=\(\frac{x}{y}\)^{\mu+3}e^{y-x} I_{\mu}(x).
\end{align*}
Analogously we get for $\mu>-1$ 
\begin{align*}
I_{\mu}(y)&>\(\frac{x}{y}\)^{\mu+3}e^{y-x}\[I_{\mu+2}(x)+\frac{2(\mu+1)}{y}I_{\mu+1}(x)\]\\
&>\(\frac{x}{y}\)^{\mu+4}e^{y-x} I_{\mu}(x).
\end{align*}
This completes the proof.
\end{proof}

\subsection{Bessel process}
We denote by $R^{(\mu)}(t)$ the Bessel process with index $\mu\in\R$ and starting from the point $x\in[0,\infty)$. It is a diffusion on $[0,\infty)$ with infinitesimal generator   given by
\begin{equation}\label{eq:generator}
\mathcal L=\frac12\frac{\partial^2 }{\partial x^2}+\frac{2\mu+1}{2x}\frac{\partial}{\partial x}.
\end{equation}

Let $p^{(\mu)}(t,x,y)$ be the transition density function (with respect to the speed measure $m(dx)=2x^{2\mu+1}$) of the  process $R^{(\mu)}(t)$. If $\mu>-1$  and $0$ is reflecting, we have
\begin{align}\label{eq:p}
p^{(\mu)}(t;x,y)&=\frac{1}{2t}(xy)^{-\mu}\exp\(-\frac{x^2+y^2}{2t}\)I_\mu\(\frac{xy}t\), \ \ \ x,y>0,\\\label{eq:p0}
p^{(\mu)}(t;0,y)&=\frac{1}{(2t)^{\mu+1}\Gamma(\mu+1)}\exp\(-y^2/2t\),\ \ \ y>0.
\end{align}
%In case  when $\mu<0$  and zero is killing, the absolute continuity property (\ref{eq:ac}) gives us $p^{(\mu)}%(t;x,y)=(xy)^{-2\mu}p^{(-\mu)}(t;x,y)$. 
%$$p(t;x,y)=\frac{1}{2t}(xy)^{-\mu}\exp\(-\frac{x^2+y^2}{2t}\)I_{|\mu|}\(\frac{xy}t\), \ \ \ x,y\in(0,1).$$
Applying (\ref{est:I}) to the above-given formulae  we obtain 
\begin{equation}\label{est:dens}
p^{(\mu)}(t;x,y)\appmu\frac{e^{-(x-y)^2/2t}}{\({yx+t}\)^{\mu+1/2}\sqrt t},\ \ \ \mu>-1.
\end{equation}
Consider now the Bessel process starting from $x>0$ and  killed when hitting the level $1$ . To distinguish two possible ways of behaviour of the process at zero, we will denote by $p_{1}^{(\mu)}(t;x,y)$, $\mu>-1$,  the transition density  of the process killed only at $1$  and by $p_{0,1}^{(\mu)}(t;x,y)$, $\mu<0$, the transition density  of the process killed  at $1$ and \mbox{at $0$} (equivalently: killed when exiting from the interval $(0,1)$).
Both of those functions may be expressed by Hunt formula, which, in case of  $p_{1}^{(\mu)}(t;x,y)$, takes the form
\begin{equation}
\label{eq:Hunt}
p_{1}^{(\mu)}(t,x,y)= p^{(\mu)}(t,x,y)-\E^x[t>T_1;p^{(\mu)}(t-T_1,1,y)],
\end{equation}
where $t>0,\ x,y\in (0,1)$. An another formula formula is also known, which follows from general theory of compact operators:
\begin{equation}\label{eq:p01series}
p_{1}^{(\mu)}(t,x,y)=(xy)^{-\mu}\sum_{n=1}^\infty \frac{J_\mu(j_{\mu,n}x)J_\mu(j_{\mu,n}y)}{\left|J_{\mu+1}(j_{\mu,n})\right|^2}  e^{-j_{\mu,n}^2t/2}.
\end{equation}
Estimates of $p_1^{(\mu)}(t,x,y)$ were recently provided in \cite{MSZ} 
\begin{align}\label{eq:estp1}
p_1^{(\mu)}(t,x,y)\appmu\(1\wedge\frac{(1-x)(1-y)}{t}\)\exp\(-\frac12j^2_{\mu,1}t\)p^{(\mu)}(t,x,y).
\end{align}
Notation and basic properties for the first exit times of the Bessel process  were presented in  Introduction. However, some complementation is needed. Since $R^{(-1/2)}(t)$ corresponds to the one-dimensional Brownian motions, we have (\cite{BS}  3.0.6 (a), (b) \mbox{p. 212})
\begin{equation}
\label{eq:p1/2}
q^{(-1/2)}_{0,1}(t,x,y)=\sum_{k=-\infty}^{\infty}\frac{\left|x-y\right|+2k}{\sqrt{2\pi}t^{3/2}}\exp\({-\frac{(\left|x-y\right|+2k)^2}{2t}}\),
\end{equation} 
where $x\in(0,1)$, $y\in\{0,1\}$. Moreover, the following scaling property holds for $a>0$ and $x\in(0,a)$, $y\in\{0,a\}$:
\begin{align}\label{eq:scalingq}
q^{(\mu)}_{a}(t,x)&=\frac{1}{a^2}\,q^{(\mu)}_{1}\(\frac t{a^2},\frac{x}{a}\),\\
q^{(\mu)}_{0,a}(t,x,y)&=\frac{1}{a^2}\,q^{(\mu)}_{0,1}\(\frac t{a^2},\frac{x}{a},\frac{y}{a}\).
\end{align}
The laws $\pr^{(\mu)}_x$ and $\pr^{(\nu)}_x$ of Bessel processes with  indices $\mu\in\R$ and $\nu\in\R$, respectively, are absolutely continuous and the corresponding Radon-Nikodym derivative is described by
\begin{equation}\label{eq:ac}
\left.\frac{d\pr^{(\mu)}_x}{d\pr^{(\nu)}_x}\right|_{\mathcal{F}_t}=\left(\frac{R^{(\nu)}(t)}{x}\right)^{\mu-\nu}\exp\left(-\frac{\mu^{2}-\nu^2}{2}\int_{0}^{t}\frac{ds}{\[R^{(\nu)}(s)\]^2}\right)\/,
\end{equation}
where $x>0$,  and the above given formula holds $\pr^{(\nu)}_x$-a.s. on $\{T_0>t\}$. Taking $\nu=-\mu$, one can see that potential theory of processes with opposite indices are closely related. In particular, for $\mu<0$ we have
\begin{align}\label{eq:acp}
p_{0,1}^{(\mu)}(t,x,y)&=(xy)^{-2\mu}p_1^{(-\mu)}(t,x,y),\\\label{eq:acq}
q^{(\mu)}_{0,1}(t,x,1)&=x^{-2\mu}q^{(-\mu)}_{1}(t,x).
\end{align}

\section{Formulae and estimates}

The discussion presented in the paper is based on the below-given theorem. The densities of exit times from  intervals $(0,1)$ and $[0,1)$ are there represented by means  of the density of the killed process. Note that if  the considered intervals were bounded away from zero, the general theory of diffusion processes could be applied. 

\begin{thm}\label{thm:qasp}
Let $x\in(0,1)$. For $\mu>-1$  we have
\begin{equation}\label{eq:qasp01}
q_1^{(\mu)}(t,x)=-\frac{\partial}{\partial y}p_1^{(\mu)}(t,x,1).
\end{equation}
Furthermore, for $\mu<0$ we have
\begin{align}\label{eq:qasp011}
q_{0,1}^{(\mu)}(t,x,1)&=-\frac{\partial}{\partial y}p_{0,1}^{(\mu)}(t,x,1),\\\nonumber
q_{0,1}^{(\mu)}(t,x,0)&=\lim_{y\rightarrow0}y^{2\mu+1}\frac{\partial}{\partial y}p_{0,1}^{(\mu)}(t,x,y)\\\label{eq:qasp010}
&=-2\mu x^{-2\mu}\, p^{(-\mu)}_{0,1}(t,x,0).
\end{align}
\end{thm}

\begin{proof}
Since the density function of the first exit time is given as a derivative of the distribution function, we get\begin{align}\nonumber
q_{1}^{(\mu)}(t,x)=&-\frac\partial{\partial t}\p^x\(T_1>t\)\\\label{aux1}
=&-2\frac\partial{\partial t}\int_0^1p_{1}^{(\mu)}(t;x,y)y^{2\mu+1}dy.
\end{align}
By (\ref{eq:j}), (\ref{eq:Jj}) and (\ref{eq:estJ}) we may differentiate  the series (\ref{eq:p01series}) term by term and the obtained function is uniformly integrable on $(0,1)$ for $t$ bounded away from zero. This allows us to switch the order of integration and differentiation in (\ref{aux1}).
Using heat equation  (cf. (\ref{eq:generator})) and integrating by parts, we get
\begin{align}\nonumber
q_{1}^{(\mu)}(t,x)=&-\int_0^1\[\frac{\partial^2}{\partial y^2}p_{1}^{(\mu)}(t;x,y)+\frac{2\mu+1}{y}\frac{\partial}{\partial y}p_{1}^{(\mu)}(t;x,y)\]y^{2\mu+1}dy\\\nonumber
=&-\int_0^1\frac{\partial}{\partial y}\[y^{2\mu+1}\frac{\partial}{\partial y}p_{1}^{(\mu)}(t;x,y)\]dy\\\label{eq:qasp}
=&-\left.\[\frac{\partial}{\partial y}p_{1}^{(\mu)}(t;x,y)\]y^{2\mu+1}\right|_{y=0}^{y=1}.
\end{align}
Applying  (\ref{eq:j}), (\ref{eq:Jj}) and (\ref{eq:estJ}) to the series (\ref{eq:p01series}) we may deduce   also that   \mbox{$\lim_{y\rightarrow0}\[\frac{\partial}{\partial y}p_{1}^{(\mu)}(t;x,y)y^{2\mu+1}\]=0$}, and this gives us (\ref{eq:qasp01}).  Assume now that $\mu<0$ and $0$ is killing. By  (\ref{eq:acp}) and (\ref{eq:acq}), the formula (\ref{eq:qasp011}) comes from the following
\begin{align*}
q^{(\mu)}_{0,1}(t,x,1)&=-\,x^{-2\mu}\left.\frac{\partial}{\partial y}p^{(-\mu)}_{1}(t,x,y)\right|_{y=1}\\
&=-\,x^{-2\mu}\left.\frac{\partial}{\partial y}\left((xy)^{2\mu}p^{(\mu)}_{0,1}(t,x,y)\right)\right|_{y=1}\\
&=-2\mu p^{(\mu)}_{0,1}(t,x,1)-\frac{\partial}{\partial y}p^{(\mu)}_{0,1}(t,x,1)\\
&=-\frac{\partial}{\partial y}p^{(\mu)}_{0,1}(t,x,1).
\end{align*}
 To prove the last part of the theorem, note that in the same way as  (\ref{eq:qasp}) we get 
\begin{align*}
\p^x\(T_{0,1}\in dt\)&=q_{0,1}^{(\mu)}(t,x,0)+q_{0,1}^{(\mu)}(t,x,1)\\
&=-\left.\[\frac{\partial}{\partial y}p_{0,1}^{(\mu)}(t;x,y)\]y^{2\mu+1}\right|_{y=0}^{y=1}.
\end{align*}
Thus, by (\ref{eq:qasp011}), we arrive at
\begin{align*}
q^{(\mu)}_{0,1}(t,x,0)=\lim_{y\rightarrow0}y^{2\mu+1}\frac{\partial}{\partial y}p_{0,1}^{(\mu)}(t;x,y).
\end{align*}
Finally, (\ref{eq:acp}) gives us
\begin{align*}
q^{(\mu)}_{0,1}(t,x,0)&=\lim_{y\rightarrow0}y^{2\mu+1}\frac{\partial}{\partial y}\left((xy)^{-2\mu}p^{(-\mu)}_{1}(t,x,y)\right)\\
&=\lim_{y\rightarrow0}x^{-2\mu}\left(-2\mu\, p^{(-\mu)}_{1}(t,x,y)+y\frac{\partial}{\partial y}p^{(-\mu)}_{1}(t,x,y)\right)\\
&=-2\mu x^{-2\mu}\, p^{(-\mu)}_{1}(t,x,0),
\end{align*}
which ends the proof.
\end{proof}
An immediate corollary of the theorem is the series formulae for the function $q^{(\mu)}_{0,1}(t,x,y)$.
\begin{cor}For $\mu<0$ and $x\in(0,1)$ we have
\begin{align}\label{eq:q011series}
q^{(\mu)}_{0,1}(t,x,1)&=x^{-\mu}\sum_{n=1}^\infty j_{-\mu,n}\frac{J_{-\mu}(j_{-\mu,n}x)}{J_{-\mu+1}(j_{-\mu,n})}  e^{-j_{-\mu,n}^2t/2},\\\label{eq:q010series}
q^{(\mu)}_{0,1}(t,x,0)&=\frac{2 x^{-\mu}}{\Gamma(-\mu)}\sum_{n=1}^\infty (j_{-\mu,n})^{-\mu}\frac{J_{-\mu}(j_{-\mu,n}x)}{\left|J_{-\mu+1}(j_{-\mu,n})\right|^2}  e^{-j_{-\mu,n}^2t/2}.
\end{align}
\end{cor}
\begin{proof}
The formula (\ref{eq:q011series}) is a consequence of (\ref{eq:acq}) and (\ref{eq:q01series}). By (\ref{eq:qasp010}) we have to calculate a limit of (\ref{eq:p01series}) as $y$ tends to zero. Since the series is uniformly convergent for every fixed $x\in(0,1)$ and $t>0$ (cf. (\ref{eq:estJ})), usage of (\ref{eq:Jat0}) leads to (\ref{eq:q010series}).    
\end{proof}
Since $p^{(-\mu)}_{1}(t,x,1)=0$, the formula (\ref{eq:estp1}) allows us to estimate the derivative $(\partial/\partial y) p_1^{(\mu)}(t;x,1)$. Combining this with Theorem \ref{thm:qasp} we may obtain estimates of densities $q^{(\mu)}_{1}(t,x)$ and $q^{(\mu)}_{0,1}(t,x,y)$.
\begin{thm}
For $\mu>-1$  we have
$$q_{1}^{(\mu)}(t,x)\appmu \frac{(1-x)(1+t)^{\mu+2}}{(x+t)^{\mu+1/2} t^{3/2}}\exp\(-\frac{(1-x)^2}{2t}-\frac12j^2_{\mu,1}t\).$$
For $\mu<0$  we have
\begin{align*}
q_{0,1}^{(\mu)}(x,t,1)&\appmu \frac{x^{-2\mu}(1-x)}{(x+t)^{-\mu+1/2}}\frac{(1+t)^{-\mu+2}}{ t^{3/2}}\exp\(-\frac{(1-x)^2}{2t}-\frac12j^2_{-\mu,1}t\),\\[10pt]
q_{0,1}^{(\mu)}(x,t,0)&\appmu \frac{x^{-2\mu}(1-x)}{1-x+t}\frac{(1+t)^{-\mu+2}}{t^{-\mu+1}}\exp\(-\frac{x^2}{2t}-\frac12j^2_{-\mu,1}t\).
\end{align*}
\end{thm}
\begin{proof}
The last assertion of the theorem follows directly from (\ref{eq:qasp010}) and (\ref{eq:estp1}). Moreover, since (\ref{eq:acq}) holds, we need only to prove the first assertion.
 We get from the previous theorem that
\begin{align*}
q_1^{(\mu)}(t,x,1)=&-\frac{\partial}{\partial y}p_{1}^{(\mu)}(t;x,1)=\lim_{y\nearrow1}\frac{p_{1}^{(\mu)}(t;x,y)}{1-y}.
\end{align*}
 We apply (\ref{eq:estp1}) and, by
$$\lim_{y\rightarrow1}\frac{1\wedge\frac{(1-x)(1-y)}{t}}{1-y}=\frac1t,$$
the proof is complete.
\end{proof}

Let $B(t)$ be the $n$-dimensional Brownian motion starting from $x\in\R^n$. By $q^n(t,x)$ we denote the density function of the exit time of Brownian motion from the unit ball $B(0,1)$ centered at the origin. Since $|B(t)|$ is a Bessel process with index $n/2-1$ (with reflecting condition at $0$ for $n=1$), we have $q^n(t,x)=q_1^{(n/2-1)}(t,x)$ and, consequently,
\begin{cor} For $x\in B(0,1)$ we have
$$q^n(t,x)\stackrel{n}{\approx} \frac{(1-|x|)(1+t)^{n/2+1}}{(|x|+t)^{(n-1)/2} t^{3/2}}\exp\(-\frac{(1-|x|)^2}{2t}-\frac12j^2_{n/2-1,1}t\).$$
\end{cor}

\section{Asymptotics}
In this section we provide precise asymptotic behaviour of  the functions $q_{1}^{(\mu)}(t,x)$ and $q_{0,1}^{(\mu)}(t,x,y)$ for small $t$. In the Theorem \ref{thm:asymp1} we describe this behaviour by means of simpler objects, i.e. the standard transition density $p^{(\mu)}(t,x,0)$ of the Bessel process reflected at zero,  and the exit time density $q_{x/4,1}^{(-1/2)}(t,x,y)$ of the one-dimensional Brownian motion from $(x/4,1)$.
\begin{thm}\label{thm:asymp1}
For $\mu>-1$ we have
\begin{align}\label{eq:auxas1}
q_{1}^{(\mu)}(t,x)&=2\frac{1-x}{t}\,p^{(\mu)}(t,x,1)\(1+O_\mu\(\frac{t}{1-x}\)\),\\\label{eq:auxas2}
q_{1}^{(\mu)}(t,x)&=\frac{q_{x/4,1}^{(-1/2)}(t,x)}{x^{\mu+1/2}}(1+O_\mu(t)),\ \ \ x\in(1/2,1),
\end{align}
and for $\mu<0$ we have
\begin{align}\label{eq:auxas3}
q_{0,1}^{(\mu)}(t,x,0)&=-2\mu x^{-2\mu}p^{(-\mu)}(t,x,0)\(1+O_\mu\(e^{-2(1-x)/t}\)\),\\\label{eq:auxas4}
q_{0,1}^{(\mu)}(t,x,0)&=-4\mu x^{-2\mu}\frac{1-x}{t}p^{(-\mu)}(t,x,0)\\\nonumber
&\ \ \ \times\(1+O_\mu\(\frac{1-x}t\)\)\(1+O_\mu(t)\),
\end{align}
where $t<t_0$ for some $t_0>0$.
\end{thm}
\begin{proof}
%Since $\frac{\partial}{\partial y}p^{(\mu)}(t-s,1,1)$ is not integrable at $s=t$ we can not take a limit under the %integral in (diffHunt) as $y\rightarrow1$. To obtain suitable formula we ...
Using Strong Markov property we obtain for $x\in(0,1)$, $\mu>-1$ and $A\subset(1,\infty)$
\begin{align*}
\E^x\[R^{\mu}(t)\in A\]&=\E^x\[T_{1}<t;R^{\mu}(t)\in A\]\\
&=\E^x\[T_{1}<t;\E^1\[R^{\mu}(t-s)\in A\]_{s=T_{1}}\],
\end{align*}
which can be expressed in the following way
$$\int_Ap^{(\mu)}(t,x,y)m(dy)=\int_0^tq_1^{(\mu)}(s,x)\int_Ap^{(\mu)}(t-s,1,y)m(dy)ds.$$
By the Fubini theorem, we deduce
$$p^{(\mu)}(t,x,y)=\int_0^tq_1^{(\mu)}(s,x)p^{(\mu)}(t-s,1,y)ds, \ \ \ x\in(0,1), y>1.$$
This let us rewrite (\ref{eq:Hunt}) into the form
\begin{align}\label{eq:rp}
p_1^{(\mu)}&(t,x,y)=p^{(\mu)}(t,x,y)-p^{(\mu)}(t,x,2-y)\\\nonumber
-&\int_0^tq_1^{(\mu)}(s,x,1)\[p^{(\mu)}(t-s,1,y)-p^{(\mu)}(t-s,1,2-y)\]ds.
\end{align}
This equality may be understood as a kind of  imitation of the reflection principle. The appearing integral plays then a role of error term and would vanish if $p^{(\mu)}(t,x,y)$ was a probability density function of Brownian motion. Furthermore, by (\ref{eq:p}) we have for $u>0$ and $x<1<y$
\begin{align}\label{eq:p/p}
\frac{p^{(\mu)}(u,x,2-y)}{p^{(\mu)}(u,x,y)}=\(\frac y{2-y}\)^{\mu}\frac{I_\mu\(x(2-y)/u\)}{I_\mu\(xy/u\)}e^{-2(1-y)/u}.
\end{align}
Hence, by Lemma \ref{lem:I/I}, we get
$$\(\frac y{2-y}\)^{2}-1<\frac{p^{(\mu)}(u,1,y)}{p^{(\mu)}(u,1,2-y)}-1<\(\frac {2-y}y\)^{2\mu+4}-1.$$
We may write it for $y>1/2$ as
$$\frac{p^{(\mu)}(u,1,y)}{p^{(\mu)}(u,1,2-y)}-1=O(1-y),$$
which gives us
\begin{align*}
\int_0^t&q^{(\mu)}(s,x,1)\(p^{(\mu)}(t-s,1,y)-p^{(\mu)}(t-s,1,2-y)\)ds\\
&=O(1-y)\int_0^tq^{(\mu)}(s,x,1)p^{(\mu)}(t-s,1,2-y)ds\\
&=O(1-y)p^{(\mu)}(t,x,2-y).
\end{align*}
This allows us to transform  (\ref{eq:rp}) to
\begin{align}\nonumber
p_1^{(\mu)}&(t,x,y)\\\label{eq:rp1}
&=p^{(\mu)}(t,x,y)\[1-\(1+O(1-y)\)\frac{p^{(\mu)}(t,x,2-y)}{p^{(\mu)}(t,x,y)}\]\\\label{eq:rp2}
&=p^{(\mu)}(t,x,y)\(1+O(1-y)\)\[1-\frac{p^{(\mu)}(t,x,2-y)}{p^{(\mu)}(t,x,y)}+O(1-y)\].
\end{align}
Using (\ref{eq:p/p}) and Lemma \ref{lem:I/I} we get
\begin{align*}
\(\frac{y}{2-y}\)^{2\mu+4}e^{-2(1-x)(1-y)/t}&<\frac{p^{(\mu)}(t,x,2-y)}{p^{(\mu)}(t,x,y)}\\
&<\(\frac{2-y}{y}\)^{2}e^{-2(1-x)(1-y)/t}.
\end{align*}
This, together with (\ref{eq:rp1}), (\ref{eq:qasp010}) and symmetry of $p^{(\mu)}_{1}(t,x,y)$,  proves the formula (\ref{eq:auxas3}). Furthermore, the above-given inequalities  give us
\begin{align*}
1-&\frac{p^{(\mu)}(t,x,2-y)}{p^{(\mu)}(t,x,y)}>1-\(\frac{2-y}{y}\)^{2}e^{-2(1-x)(1-y)/t}\\
&=1-e^{-2(1-x)(1-y)/t}+\(1-\(\frac{2-y}{y}\)^{2}\)e^{-2(1-x)(1-y)/t}\\
&=2\frac{(1-x)(1-y)}t+O_\mu\(\(\frac{(1-y)(1-x)}{t}\)^2\)+O_\mu(1-y),
\end{align*}
and similarly
\begin{align*}
1-&\frac{p^{(\mu)}(t,x,2-y)}{p^{(\mu)}(t,x,y)}<1-\(\frac{y}{2-y}\)^{2\mu+4}e^{-2(1-x)(1-y)/t}\\
&=2\frac{(1-x)(1-y)}t+O_\mu\(\(\frac{(1-y)(1-x)}{t}\)^2\)+O_\mu(1-y),
\end{align*}
which can be generally written as 
\begin{align*}
1-&\frac{p^{(\mu)}(t,x,2-y)}{p^{(\mu)}(t,x,y)}\\
&=2\frac{(1-x)(1-y)}t+O_\mu\(\(\frac{(1-y)(1-x)}{t}\)^2+1-y\).
\end{align*}
Combining this with (\ref{eq:rp2}) we arrive at
\begin{align*}
p_1^{(\mu)}&(t,x,y)=\,p^{(\mu)}(t,x,y)\(1+O_\mu(1-y)\)\\
&\,\times\[2\frac{(1-x)(1-y)}t+O_\mu\(\(\frac{(1-y)(1-x)}{t}\)^2+1-y\)\].
\end{align*}
This equality implies formulae (\ref{eq:auxas1}) and (\ref{eq:auxas4}). Indeed, we have
\begin{align*}
q_1^{(\mu)}(t,x)&=\lim_{y\nearrow1}\frac{p_{1}^{(\mu)}(t,x,y)}{1-y}=\(2\frac{1-x}{t}+O_\mu(1)\)p^{(\mu)}(t,x,1).
\end{align*}
and
\begin{align*}
q_{0,1}^{(\mu)}(t,x,0)&=-2\mu x^{-2\mu}p_{1}^{(\mu)}(t,0,x)\\
&=-2\mu\(1+O\(1-x\)\)2\frac{1-x}{t}p^{(\mu)}(t,0,x)\\
&\ \ \ \ \times\(1+O_\mu\(\frac{1-x}t+t\)\)\\
&=-4\mu\frac{1-x}{t}p^{(\mu)}(t,x,0)\(1+O\(\frac{1-x}t\)\)\(1+O_\mu(t)\).
\end{align*}
It is showed in \cite{MSZ} that there exists $t_0>0$ such that 
\begin{equation}\label{eq:pasp1/4}p_{1}^{(\mu)}(t,x,y)=\(1+O_\mu\(e^{-1/64t}\)\)p_{x/4,1}^{(\mu)}(t,x,y),
\end{equation}
whenever $x,y\in(1/2,1)$ and $t<t_0$. By the absolute continuity (\ref{eq:ac}) of distributions of Bessel processes with different  indices we get
\begin{align*}
	\E^x&\left[t<T^{(\mu)}_{x/4,1}, R^{(\mu)}(t)\in A\right]\\
&=  \E^x\bigg[t<T^{(-1/2)}_{x/4,1}, R^{(-1/2)}(t)\in A;\\
&\ \ \ \ \ \ \ \ \ \left(\frac{R^{(-1/2)}(t)}{x}\right)^{\mu+1/2}\exp\left(-\frac{\mu^2-1/4}{2}\int_0^t \frac{ds}{[R^{(-1/2)}(s)]^2}\right)\bigg]\/.
\end{align*}
Since $R^{(-1/2)}(s)>1/8$ on $\{t<\tau^{-1/2}_{(x/4,1)}\}$, we have
$$\int_0^t \frac{ds}{[R^{(-1/2)}(s)]^2}<64t,$$
and consequently
\begin{align*}
\exp&\left(-\frac{\mu^2-1/4}{2}\int_0^t \frac{ds}{[R^{(-1/2)}(s)]^2}\right)\\
&=1+\[\exp\left(-\frac{\mu^2-1/4}{2}\int_0^t \frac{ds}{[R^{(-1/2)}(s)]^2}\right)-1\]\\
&=1+O_\mu(t).
\end{align*}
Combining this with (\ref{eq:pasp1/4})  we obtain 
$$p^{(\mu)}_{1}(t,x,y)=(1+O_\mu(t))\left(\frac{y}{x}\right)^{\mu+1/2}p^{(-1/2)}_{x/4,1}(t,x,y).$$
Thus, formulae  (\ref{eq:qasp01}) and  (\ref{eq:qasp011}), scaling property and the shift-invariance of one-dimensional Brownian motion give us
 \begin{align*}
q_{1}^{(\mu)}(t,x)&=(1+O(t))\frac{\partial}{\partial y}\[\left(\frac{y}{x}\right)^{\mu+1/2}p^{(-1/2)}_{x/4,1}(t,x,y)\]_{y=1}\\
&=\frac{(1+O(t))}{x^{\mu+1/2}}\[(\mu+1/2)p^{(-1/2)}_{x/4,1}(t,x,1)+\frac{\partial}{\partial y}p^{(-1/2)}_{x/4,1}(t,x,1)\]\\
&=\frac{(1+O(t))}{x^{\mu+1/2}}q_{x/4,1}^{(-1/2)}(t,x,1).
\end{align*}
This completes the proof.
\end{proof}
We are now prepared to express the asymptotic behaviour explicitly.

\begin{thm}For $\mu>-1$  we have
\begin{align}\label{eq:x/t->00}
q_{1}^{(\mu)}(t,x,1)&=\frac{1-x}{\sqrt{2\pi t^3}}\frac{e^{-(1-x)^2/2t}}{x^{\mu+1/2}}\(1+O_\mu\(\frac t{x}\)\),\\\label{eq:x/t->0}
q_{1}^{(\mu)}(t,x,1)&=\frac{1-x}{t^{\mu+1}}\frac{e^{-(1-x)^2/2t}}{2^\mu\Gamma(\mu+1)}\(1+O_\mu\(\(\frac {x}t\)^2\)\)\(1+O_\mu\(t\)\).
\end{align}
If $\mu<0$, then we have 
\begin{align}\label{eq:(1-x)/t->00}
q_{0,1}^{(\mu)}(t,x,0)&=\frac{2 x^{-2\mu}}{(2t)^{-\mu+1}\Gamma(-\mu)}e^{-x^2/2t}\(1+O_\mu\(e^{-2(1-x)/t}\)\),\\\label{eq:(1-x)/t->0}
q_{0,1}^{(\mu)}(t,x,0)&=\frac{8(1-x) x^{-2\mu}}{(2t)^{-\mu+2}\Gamma(-\mu)}e^{-x^2/2t}\(1+O_\mu\(\frac{1-x}t\)\)\(1+O_\mu(t)\),
\end{align}
where $t<t_0$ for some $t_0>0$.
\end{thm}
\begin{rem*}
Asymptotic behaviour of $q_{0,1}^{(\mu)}(t,x,1)$  follows from the above-given theorem and the relationship (\ref{eq:acq}).
\end{rem*}
\begin{proof}
To obtain (\ref{eq:(1-x)/t->00}) and (\ref{eq:(1-x)/t->0}), it is enough to apply (\ref{eq:p0}) to (\ref{eq:auxas3}) and (\ref{eq:auxas4}), respectively. Furthermore, if $x/t\rightarrow0$, then $t/(1-x)\approx t$. Hence, by (\ref{eq:auxas1}) and (\ref{eq:asympI0}), we get (\ref{eq:x/t->0}). What has left is to prove (\ref{eq:x/t->00}). Rewriting (\ref{eq:p1/2}) for $y=1$ we get
\begin{align*}
q^{(-1/2)}_{0,1}&(t,x,1)\\
=&\frac{1-x}{\sqrt{2\pi}t^{3/2}}\exp\({-\frac{(1-x)^2}{2t}}\)\[\sum_{k=-\infty}^{\infty}\exp\({-\frac{2k(k+1-x)}{t}}\)\right.\\
&\left.+\sum_{k=1}^{\infty}\frac{2k}{1-x}\(e^{-2k(k+1-x)/t}-e^{-2k(k-(1-x))/t}\)\]\\
=:&\frac{1-x}{\sqrt{2\pi}t^{3/2}}\exp\({-\frac{(1-x)^2}{2t}}\)\[S_1+S_2\].
\end{align*}
It is clear that $S_1=1+O_\mu(e^{-2x/t})$. By inequality $\sinh r<re^r$, $r>0$, we get
\begin{align*}
0<S_2&=\sum_{k=1}^{\infty}\frac{2k}{1-x}\(e^{-2k(k+1-x)/t}-e^{-2k(k-(1-x))/t}\)\\
&<\sum_{k=1}^{\infty}\frac{4k}{1-x}e^{-2k^2/t}\sinh\(2k(1-x)/t)\)\\
&<\frac{e^{-2x/t}}{t}\sum_{k=1}^{\infty}8k^2e^{-2k(k-1)/t}.
\end{align*}
Hence, for $1/2<x<1$ we obtain $S_2=O_\mu(e^{-x/t})$, and we arrive at
$$q^{(-1/2)}_{0,1}(t,x,1)=\frac{1-x}{\sqrt{2\pi}t^{3/2}}\(1+O_\mu(e^{-x/t})\).$$ 
Thus, using (\ref{eq:auxas2}),  scaling property and the shift-invariance of one-dimensional Brownian motion, we obtain
\begin{align*}
q_{1}^{(\mu)}(t,x)&=\frac{1-x}{\sqrt{2\pi}t^{3/2}}\frac{e^{-(1-x)^2/2t}}{x^{\mu+1/2}}\(1+O_\mu(e^{-x(1-x/4)/t})\)(1+O_\mu(t))\\
&=\frac{1-x}{\sqrt{2\pi}t^{3/2}}\frac{e^{-(1-x)^2/2t}}{x^{\mu+1/2}}\(1+O_\mu\(\frac tx\)\),\ \ \ \ \ 1/2<x<1.
\end{align*}
In addition, (\ref{eq:auxas1}) and (\ref{eq:asympI0}) give us for $0<x\leq1/2$
\begin{align*}
q_{0,1}^{(\mu)}(t,x,1)&=\frac{1-x}{\sqrt{2\pi}t^{3/2}}\frac{e^{-(1-x)^2/2t}}{x^{\mu+1/2}}\(1+O_\mu\(\frac tx\)\)\(1+O_\mu\(t\)\).
\end{align*} Since $t/x>t$, the above-given formula is equivalent to (\ref{eq:x/t->00}). This ends the proof.
\end{proof}
\section*{Acknowledgment} 
The author is very grateful to M. Kwa\'snicki for some ideas concerning proof of Theorem \ref{thm:qasp} and to   J. Ma{\l}ecki for several helpfull suggestions.

\bibliography{bibliography}
\bibliographystyle{plain}

\end{document}